\documentclass[a4paper,reqno]{amsart}

\usepackage{amssymb}
\usepackage{pgfmath}
\usepackage{latexsym}
\usepackage{amsmath}
\usepackage{amsthm, lmodern, color}
\usepackage{bbm}

\usepackage[normalem]{ulem}
\definecolor{darkred}{rgb}{0.6,0.2,0.2}
\usepackage[colorlinks=true,linkcolor=blue, citecolor=darkred]{hyperref}

\usepackage{pgfplots}
\pgfplotsset{compat=1.15}
\usepackage{mathrsfs}
\usetikzlibrary{arrows}
\definecolor{ududff}{rgb}{0.30196078431372547,0.30196078431372547,1}
\definecolor{zzttqq}{rgb}{0.6,0.2,0}
\definecolor{xdxdff}{rgb}{0.49019607843137253,0.49019607843137253,1}
\definecolor{ududff}{rgb}{0.30196078431372547,0.30196078431372547,1}
\usepackage{hyperref}
\usepackage{xcolor}
\usepackage{amsmath}
\definecolor{darkblue}{rgb}{0.1,0.1,0.6}
\newcommand{\Hm}[1]{\leavevmode{\marginpar{\tiny%
			$\hbox to 0mm{\hspace*{-0.5mm}$\leftarrow$\hss}%
			\vcenter{\vrule depth 0.1mm height 0.1mm width \the\marginparwidth}%
			\hbox to
			0mm{\hss$\rightarrow$\hspace*{-0.5mm}}$\\
			\relax\raggedright #1}}}

\def\phi{\varphi}

      \def\dC{{\mathbb C}}

   \def\dN{{\mathbb N}}   
      \def\dR{{\mathbb R}}

      \def\cL{{\mathcal L}}

\DeclareMathOperator\Real{{\text{\rm Re}}}

\DeclareMathOperator{\R}{\mathbb{R}}

\def\dd{\,\mathrm{d}}

\newcommand{\beq}{\begin{equation} \begin{split}}
\newcommand{\eeq}{\end{split} \end{equation}}

\newcommand\Omg{\Omega}

\renewcommand\and{\qquad\text{and}\qquad}

\newcommand\sm{\setminus}

\newcommand{\comm}[1]{}

\def\sfH{\mathsf{H}}

\def\bm1{\mathbbm{1}}

\def\p{\partial}

\def\arr{\rightarrow}

\def\tt{\theta}
\def\aa{\alpha}
\def\lm{\lambda}

\def\p{\partial}

\def\sfH{\mathsf{H}}

\def\sfP{\mathsf{P}}
\def\dd{{\,\mathrm{d}}}

\def\sfP{\mathsf{P}}

\def\sfV{\mathsf{V}}

\newcommand{\beu}{\begin{equation*}}
\newcommand{\eeu}{\end{equation*}}
\newcommand{\besu}{\begin{equation*}
\begin{aligned}}
\newcommand{\eesu}{\end{aligned}
\end{equation*}}
\newcommand{\bes}{\begin{equation}
\begin{aligned}}
\newcommand{\ees}{\end{aligned}
\end{equation}}

\newcommand\frh{\mathfrak h}

\newcommand\ov{\overline}

\newcommand{\dom}{\mathrm{dom}\,}

\definecolor{darkblue}{rgb}{0.2,0.2,0.6}

\def\sfA{{\mathsf A}}
\def\sfL{{\mathsf L}}
\def\sfM{{\mathsf M}} 

\renewcommand\setminus{\mathbin{\mathpalette\rsetminusaux\relax}}
\newcommand\rsetminusaux[2]{\mspace{-4mu}
  \raisebox{\rsmraise{#1}\depth}{\rotatebox[origin=c]{-20}{$#1\smallsetminus$}}
 \mspace{-4mu}
}
\newcommand\rsmraise[1]{%
  \ifx#1\displaystyle .8\else
    \ifx#1\textstyle .8\else
      \ifx#1\scriptstyle .6\else
        .45%
      \fi
    \fi
  \fi}

\makeatletter
\newcommand*{\rom}[1]{\expandafter\@slowromancap\romannumeral #1@}
\makeatother

\newtheorem{theorem}{Theorem}[section]
\newtheorem{proposition}[theorem]{Proposition}
\newtheorem{corollary}[theorem]{Corollary}
\newtheorem{lemma}[theorem]{Lemma}

\theoremstyle{definition}

\theoremstyle{definition}

\numberwithin{equation}{section}

\title[The second positive Neumann eigenvalue for parallelograms]{A note on optimization of the second positive Neumann eigenvalue for parallelograms}

\author[V.~Lotoreichik]{Vladimir Lotoreichik}
\address{Department of Theoretical Physics,
Nuclear Physics Institute, Czech Academy of Sciences, 25068, \v{R}e\v{z}, Czech Republic}
\email{lotoreichik@ujf.cas.cz}

\author[J.~Rohleder]{Jonathan Rohleder}
\address{Matematiska institutionen \\ Stockholms universitet \\
106 91 Stockholm \\
Sweden}
\email{jonathan.rohleder@math.su.se}

\begin{document}

\begin{abstract}
It has recently been conjectured by Bogosel, Henrot, and Michetti that the second positive eigenvalue of the Neumann Laplacian is maximized, among all planar convex domains of fixed perimeter, by the rectangle with one edge length equal to twice the other. In this note we prove that this conjecture is true within the class of parallelogram domains.
\end{abstract}

\maketitle

\section{Introduction}

On a bounded, sufficiently regular domain $\Omega \subset \R^2$ the eigenvalue problem for the Laplacian with Neumann boundary conditions is given by
\begin{align*}
 \begin{cases}
 - \Delta \psi = \mu \psi & \text{in}~\Omega, \\
 \hspace*{2.0mm} \partial_\nu \psi = 0 & \text{on}~\partial \Omega,
 \end{cases}
\end{align*}
where $\partial_\nu \psi$ denotes the derivative of $\psi$ in the direction of the exterior unit normal field on the boundary $\partial \Omega$. This problem has a sequence of eigenvalues $0 = \mu_1 (\Omega) < \mu_2 (\Omega) \leq \mu_3 (\Omega) \leq \dots$, which we count according to their respective multiplicities. A classical result due to Szeg\H{o} (and Weinberger, in higher dimensions) states that among all $\Omega$ of fixed area, $\mu_2 (\Omega)$ is maximized by the disk, see e.g.\ \cite[Theorem 7.1.1]{H06}. About half of a decade later it was shown by Girouard, Nadirashvili, and Polterovich \cite{GNP09} that the unique maximizer of $\mu_3 (\Omega)$ under an area constraint is the disjoint union of two disks; this was generalized to higher dimensions by Bucur and Henrot \cite{BH19}. A related question which has attracted considerable interest in recent years is the maximization of eigenvalues under a perimeter constraint, i.e.\ determining the quantity
\begin{align}\label{eq:problem}
 \sup \left\{ \mu_k (\Omega) |\partial \Omega|^2 : \Omega \subset \R^2~\text{convex} \right\}
\end{align}
for some $k \geq 2$, where $|\partial \Omega|$ denotes the perimeter of $\Omega$, and finding optimal shapes if they exist. It is well known that this supremum is infinity if the convexity assumption on $\Omega$ is omitted, see, e.g. the construction in~\cite[Proposition 3.3]{HLL24}. On the other hand, it was recently shown in \cite[Theorem 2.6]{BHM24} that this supremum is always attained, i.e.\ a maximizing shape exists for each $k$. 

The case $k = 2$ has seen considerable progress very recently. In fact, it has been noticed in~\cite[Problem 9.2]{LS09} that an optimal shape for the maximization problem~\eqref{eq:problem} with $k = 2$ cannot be the disk
and an open problem about the optimal shape has been posed. It was stated as a conjecture by R.~Laugesen in~\cite{ABLW09}
that $\mu_2 (\Omega) |\p\Omega|^2 \leq 16 \pi^2$ holds for all convex $\Omega$ and that equality holds if and only if $\Omega$ is a square or an equilateral triangle. This has been verified for special classes of convex domains such as triangles \cite{LS09} and parallelograms \cite{HLL24}, see also \cite{LR24}. However, it has been a recent breakthrough that the conjecture is true within the class of all convex domains with two (not necessarily orthogonal) axes of symmetry \cite{HLL24}. 

In this note we focus on the case $k = 3$ of the second positive Neumann eigenvalue. It was conjectured and numerically verified in~\cite[Section 4]{BHM24} that among all planar convex domains of fixed perimeter the third Neumann eigenvalue is maximised by the rectangle one side of which is twice as long as the other. Relying on the fact that the third Neumann eigenvalue of the rectangle can be explicitly computed, it is not hard to see that this conjecture is equivalent to the upper bound
\begin{equation}\label{eq:mu3}
	\mu_3(\Omg) |\partial \Omega|^2 \le 36\pi^2
\end{equation}
for any bounded convex domain $\Omg\subset\dR^2$. Inspired by this conjecture, in this note we verify the bound~\eqref{eq:mu3} among parallelograms. This class of domains is a natural choice due to the fact that it includes all rectangles as a subclass and thus the conjectured optimizer is contained. In fact, we prove that the unique maximizing domain among parallelograms is indeed the rectangle with edge length ratio 2:1. This is done by a variational argument constructing appropriate trial functions. In our analysis, we parametrize the parallelogram by two real parameters. After that we split the parameter space into four disjoint regions, where in  three of them we employ appropriate trial subspaces, while in the remaining region we make use of a general upper bound due to Kr\"oger.

In this context, we also refer to the overview~\cite{LS17} on eigenvalue optimization for the Laplacian with Dirichlet and Neumann boundary conditions on special domains such as triangles or rhombi. More recent contributions address optimization of the Robin eigenvalues on special domains such as triangles~\cite{KLV23} or quadrilaterals~\cite{CL24}. 

As for the structure of this note, in Section \ref{sec:trafo} we provide an alternative variational principle for the eigenvalues of the Neumann Laplacian on a parallelogram based on a linear transformation onto a rectangle. In Section \ref{sec:trial} we compute the corresponding Rayleigh quotients for a collection of selected trial functions. Finally, in Section \ref{sec:main} we formulate and prove the main result of this note, Theorem \ref{thm:mu3}.

\section{The Neumann Laplacian on parallelograms}\label{sec:trafo}

The aim of this section is to derive an alternative variational principle for the eigenvalues of the Neumann Laplacian on a parallelogram. Let us denote by $\Omega_{c,d} \subset \R^2$ the (open) parallelogram spanned by the vectors $(1, 0)^\top$ and $(c, d)$, see Figure~\ref{fig1}, where $c \geq 0, d > 0$ and $c^2 + d^2 \leq 1$. 
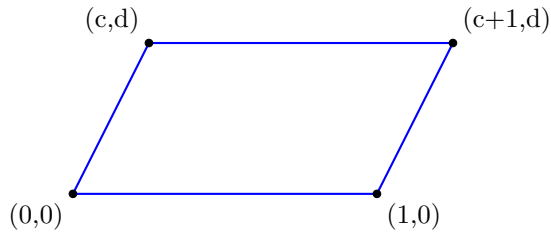
\begin{figure}[h]
\begin{tikzpicture}
% Define the vertices of the parallelogram
\coordinate (A) at (0, 0);  % Bottom-left corner
\coordinate (B) at (4, 0);  % Bottom-right corner
\coordinate (C) at (5, 2);  % Top-right corner
\coordinate (D) at (1, 2);  % Top-left corner
% Draw the edges of the parallelogram
\draw[thick, blue] (A) -- (B) -- (C) -- (D) -- cycle;
\node[draw, fill=black, circle, inner sep=1pt] at (A) {};
\node[draw, fill=black, circle, inner sep=1pt] at (B) {};
\node[draw, fill=black, circle, inner sep=1pt] at (C) {};
\node[draw, fill=black, circle, inner sep=1pt] at (D) {};
% Label the vertices
\node[below left] at (A) {(0,0)};
\node[below right] at (B) {(1,0)};
\node[above right] at (C) {(c+1,d)};
\node[above left] at (D) {(c,d)};
\end{tikzpicture}
\caption{The parallelogram $\Omg_{c,d}$ with vertices having 
coordinates $(0,0)$, $(c,d)$, $(c+1,d)$, and $(1,0)$.}
\label{fig1}
\end{figure}
The vertices of this parallelogram, enumerated in the clockwise order, have coordinates $(0,0)$, $(c,d)$, $(c+1,d)$, and $(1,0)$, respectively. It is clear that in order to verify~\eqref{eq:mu3} among all parallelograms it suffices to verify this bound among all parallelograms $\Omega_{c,d}$ satisfying the mentioned restrictions on $c$ and $d$; indeed, any parallelogram can be scaled in such a way that its largest side has length $1$, and then it can by rigid motion and reflection be brought into the above form. However, the expression $\mu_3 (\Omega) |\partial \Omega|^2$ is invariant under scaling, rigid motion and reflection.

Recall that the self-adjoint Neumann Laplacian $-\Delta_{\rm N}^{\Omg_{c,d}}$ in the Hilbert space $L^2(\Omg_{c,d})$ can be introduced via the first representation theorem using the closed, densely defined and non-negative quadratic form
\[
	\frh_{\rm N}^{\Omg_{c,d}}[u] := \int_{\Omega_{c,d}} |\nabla u|^2 \dd x \dd y,\qquad \dom \frh_{\rm N}^{\Omg_{c,d}} := H^1(\Omg_{c,d}).
\]
The spectrum of this operator is purely discrete thanks to compactness of the embedding of $H^1(\Omg_{c,d})$ into $L^2(\Omg_{c,d})$. We denote by $0 = \mu_1 (\Omega_{c,d}) < \mu_2 (\Omega_{c,d}) \leq \mu_3 (\Omega_{c,d}) \leq \dots$ the eigenvalues of $-\Delta_{\rm N}^{\Omg_{c,d}}$, counted with multiplicities.

We will now construct an operator acting on the rectangle $\Omega_{0,d}$ which is unitarily equivalent to $- \Delta_{\rm N}^{\Omega_{c,d}}$. For this, consider the linear mapping
\[
	\Phi\colon \Omega_{0,d} \to \Omega_{c,d},\qquad \Phi(x,y) := \begin{pmatrix}
		x +y\frac{c}{d} \\ y \end{pmatrix}.
\]
It is easy to see that $\Phi$ provides a bijection from the rectangle $\Omega_{0,d}$ onto the parallelogram $\Omega_{c,d}$. In the next lemma, we compute the quadratic form of the operator in $L^2(\Omg_{0,d})$, which is unitarily equivalent to the Neumann Laplacian on $\Omg_{c,d}$ via the transformation $\Phi$.

\begin{lemma}
	The Neumann Laplacian $-\Delta_{\rm N}^{\Omg_{c,d}}$ on $\Omg_{c,d}$ is unitarily equivalent to
	the self-adjoint operator $\sfH_{c,d}$ in the Hilbert space $L^2(\Omg_{0,d})$ associated with the closed, non-negative,
	and densely defined quadratic form
	\begin{equation}\label{eq:formcd}
	\begin{aligned}
		\frh_{c,d}[u] & := \int_0^d\int_0^1\left[
		\left(1+\frac{c^2}{d^2}\right)|\p_x u|^2 + |\p_y u|^2 -
		\frac{2c}{d} \Real \big(\p_x u \,\ov{\p_y u} \big)\right]\dd x \dd y,\\
		\dom\frh_{c,d}& := H^1(\Omg_{c,d}).
	\end{aligned}	
	\end{equation}
\end{lemma}

\begin{proof}
	The Jacobian matrix of the mapping $\Phi$ can be computed explicitly and is given by
	\[
		D\Phi = \begin{pmatrix} 1 & \frac{c}{d}\\
			0 & 1\end{pmatrix}.
	\]
	The determinant of this matrix is equal to one. Therefore, the mapping 
	\[
		\sfV\colon L^2(\Omg_{c,d})\arr L^2(\Omg_{0,d}),\qquad \sfV u := u\circ\Phi
	\]
	is unitary. The self-adjoint operator $\sfV(-\Delta_{\rm N}^{\Omg_{c,d}})\sfV^{-1}$ in $L^2(\Omg_{0,d})$ corresponds	to the closed, non-negative, densely defined quadratic form
	\begin{equation}\label{eq:formPhi}
		H^1(\Omg_{0,d})\ni u \mapsto \int_{\Omega_{c,d}} |\nabla (u \circ \Phi^{-1})|^2 \dd x \dd y,
	\end{equation}
	where we implicitly have used that $\sfV$ maps $H^1(\Omg_{c,d})$ bijectively onto $H^1(\Omg_{0,d})$.
	It remains to verify that the quadratic form~\eqref{eq:formPhi} is equal
	to the form $\frh_{c,d}$ in~\eqref{eq:formcd}. In fact, by a direct computation using the chain rule, the substitution formula, and the fact that $\det D \Phi = 1$, we obtain
	\begin{align*}
	 \int_{\Omega_{c,d}} |\nabla (u \circ \Phi^{-1})|^2 \dd x \dd y & = \int_{\Omega_{c,d}} \left| (D \Phi^{-1})^\top (\nabla u) \circ \Phi^{-1} \right|^2 \dd x \dd y \\
	 & = \int_{\Omega_{0,d}} \left| (D \Phi^{-1})^\top \nabla u \right|^2 \dd x \dd y \\
	 & = \int_{\Omega_{0,d}} \left| \binom{\partial_{x} u}{- \frac{c}{d} \partial_{x} u + \partial_y u} \right|^2 \dd x \dd y,
	\end{align*}
    which leads to \eqref{eq:formcd}.
\end{proof}

The unitary equivalence stated in the above lemma and the classical min-max principle (see e.g.~\cite[Theorem~1.28]{FLW}) yield the following variational characterisation.

\begin{corollary}\label{cor:RR}
The eigenvalues of $- \Delta_{\rm N}^{\Omega_{c,d}}$ are given by
\begin{equation*}
	\mu_k(\Omg_{c,d}) = 
	\min_{\begin{smallmatrix} \cL \subset H^1(\Omg_{0,d})\\ \dim\cL = k \end{smallmatrix} } \max_{u\in\cL\sm\{0\}}\frac{\frh_{c,d}[u]}{\int_{\Omega_{0,d}} |u|^2 \dd x \dd y},\qquad k\in\dN,
\end{equation*}
where the minimum is taken over $k$-dimensional linear subspaces of $H^1(\Omg_{0,d})$.
\end{corollary}

\section{Application of the Rayleigh-Ritz principle}\label{sec:trial}

% Our aim is to get an upper bound on $\mu_3(\Omg_{c,d})$ in terms of parameters $c$ and $d$ which would imply~\eqref{eq:mu3} for parallelograms. The employed technique relies on covering without holes the parameter space $c\in[0,1]$, $d\in(0,1)$ with $c^2+d^2 < 1$ by two bounds, which are stronger than~\eqref{eq:mu3}
% in two overlapping regions. In the region of the parameter space close to the point $c = 0$ and $d=\frac12$, corresponding to the expected optimizer, we apply the Rayleigh-Ritz principle
% to the quadratic form $\frh_{c,d}$ in~\eqref{eq:formcd} tested on the eigenfunctions of the rectangle $(0,1)\times(0,d)$ corresponding to its lowest five eigenvalues. The bound based on this method covers a significant part of the parameter space including the neighbourhood of the point $c = 0$, $d = \frac12$. The remaining part of the parameter space can be  covered by the bound obtained by Kr\"oger in~\cite[Theorem 1]{K99}.
 
%One can deduce
%from~\cite[Corollary 1.32]{FLW23} the following. Let $\sfT$ be a self-adjoint and lower-semibounded operator in a Hilbert space $\cH$ with quadratic form $\frt$. Let $\sfP$ be an $N$-dimensional orthogonal projection in $\cH$ such that $\ran\sfP \subset\dom\frt$. The operator $\sfP\sfT\sfP$ is thus well defined and let us denote by $\sfT_\sfP$ the restriction of $\sfP\sfT\sfP$ to $\ran\sfP$. Then, for all $k\in\{1,\dots, N\}$, the following inequality $\mu_k(\sfT) \le \mu_k(\sfT_\sfP)$ holds.

In this section we apply the Rayleigh--Ritz principle, based on Corollary \ref{cor:RR}, to appropriate trial functions. These trial functions are chosen with the goal to imply \eqref{eq:mu3} on the parallelogram $\Omega_{c,d}$ for large parts of the parameter space $c \geq 0, d > 0, c^2 + d^2 \leq 1$. The trial functions we are going to use are the normalized eigenfunctions of the Neumann Laplacian on the rectangle $\Omega_{0,d}$ corresponding to its first 5 eigenvalues. In Section \ref{sec:main} this will turn out useful to prove \eqref{eq:mu3} for a large region of points $(c, d)$ in the parameter space and specifically in a neighborhood of the parameter $(c, d) = (0, \tfrac{1}{2})$ corresponding to the optimizing domain.

We are going to express the following spectral bounds in terms of the eigenvalues of the two Hermitian matrices
\begin{align}\label{eq:LM}
	\sfL_{c,d} := \begin{pmatrix}
		\pi^2\left(1+\frac{c^2}{d^2}\right) & -\frac{8c}{d^2}\\
		-\frac{8c}{d^2} & \frac{\pi^2}{d^2}
		\end{pmatrix} \;\; \text{and} \;\;
	\sfM_{c,d} := \begin{pmatrix}
        \pi^2\left(1+\frac{c^2}{d^2}\right) +\frac{\pi^2}{d^2} &
				-\frac{32\sqrt{2}c}{3d^2}\\
		-\frac{32\sqrt{2}c}{3d^2} & 4\pi^2\left(1+\frac{c^2}{d^2}\right)		
		\end{pmatrix}.
\end{align}
For any Hermitian matrix $\sfM$ of size $N\times N$, where $N\in\dN$, we use the convention to denote its eigenvalues by $\lambda_1 (\sfM) \leq \lambda_2 (\sfM) \leq \dots$ with multiplicities taken into account.

\begin{proposition}\label{prop:LM}
	The inequality
	\[
		\mu_3(\Omg_{c,d}) \le \lm_2(\sfL_{c,d}\oplus \sfM_{c,d})
	\]
	holds for all $c \geq 0$ and all $d > 0$.
\end{proposition}

\begin{proof}
	Let us consider the following five functions on
	the rectangle $\Omg_{0,d}$
	\[
	\begin{aligned}
		\psi_1(x,y) & = \frac{1}{\sqrt{d}},\quad
		&\psi_2(x,y) & = \sqrt{\frac{2}{d}}\cos(\pi x), \\
		\psi_3(x,y) & = \sqrt{\frac{2}{d}}\cos\left(\frac{\pi y}{d}\right), \quad
		& \psi_4(x,y) & = \frac{2}{\sqrt{d}}\cos(\pi x)\cos\left(\frac{\pi y}{d}\right), \\
		\psi_5(x,y) & = \sqrt{\frac{2}{d}}\cos(2\pi x).
	\end{aligned}	
	\]
	It is straightforward to see that this family of functions is orthonormal. Let $\sfP$ be the orthogonal projection in the Hilbert space $L^2(\Omg_{0,d})$
	onto the subspace $\cL:= {\rm span}\,\{\psi_1,\psi_2,\dots,\psi_5\}$ of $H^1(\Omg_{0,d})$. By~\cite[Corollary 1.32]{FLW} we have
	\begin{equation}\label{eq:Ray-Ritz}
		\mu_k(\Omg_{c,d}) \le \lm_k(\sfP\sfH_{c,d}\sfP) = \lambda_k (\sfA),\qquad 1\le k \le 5,
	\end{equation}
	where the operator $\sfP\sfH_{c,d}\sfP$ has been identified with a $5\times 5$ symmetric matrix $\sfA = (\aa_{ij})_{i,j=1}^5$ via 
	\begin{align*}
	 (\sfA \xi, \xi)_{\mathbb C^5} = \frh_{c,d}[\psi] \quad \text{if} \quad \psi = \sum_{i = 1}^5 \xi_i \psi_i
	\end{align*}
    for $\xi \in \mathbb C^5$. The entries of $\sfA$ are given by the formula
	\begin{equation*}
		\aa_{ij} := \frh_{c,d}[\psi_i,\psi_j],\qquad 1 \le i,j\le 5.
	\end{equation*}
	By a direct computation, we find that the diagonal entries of the matrix $\sfA$ are given by 
	\[
	\begin{aligned}
		\aa_{11} &= 0,\\
		\aa_{22} &=  \frac{2\pi^2}{d}\int_0^d
		\int_0^1\left(1+\frac{c^2}{d^2}\right)\sin^2(\pi x)\dd x\dd y  =
		\pi^2\left(1+\frac{c^2}{d^2}\right),\\
		\aa_{33} &= \frac{2\pi^2}{d^3}\int_0^d
		\int_0^1\sin^2\left(\frac{\pi y}{d}\right)\dd x\dd y  =
		\frac{\pi^2}{d^2},\\
		\aa_{44} & = \frac{4\pi^2}{d}\int_0^d
		\int_0^1
		\left[\left(1+\frac{c^2}{d^2}\right)\cos^2\left(\frac{\pi y}{d}\right)\sin^2(\pi x)+\frac{1}{d^2}\sin^2\left(\frac{\pi y}{d}\right)\cos^2(\pi x)\right]\dd x\dd y\\
		&  =\pi^2\left(1+\frac{c^2}{d^2}\right) +\frac{\pi^2}{d^2},\\
		\aa_{55} &= \frac{8 \pi^2}{d}\int_0^d
		\int_0^1\left(1+\frac{c^2}{d^2}\right)\sin^2(2\pi x)\dd x\dd y  = 4\pi^2\left(1+\frac{c^2}{d^2}\right).\\
	\end{aligned}
	\]	
	The non-zero off-diagonal entries can be computed as follows
	\[
	\begin{aligned}
		\aa_{23} = \aa_{32} & =  -\frac{2\pi^2c}{d^3}
		\int_0^d\int_0^1 \sin(\pi x) \sin\left(\frac{\pi y}{d}\right)\dd x \dd y = -\frac{8c}{d^2},\\
		\aa_{45} = \aa_{54} & = 
		\frac{2\sqrt{2}}{d}
		\int_0^d\int_0^1\bigg[
		\left(1+\frac{c^2}{d^2}\right)
		2\pi^2\sin\left(2\pi x\right)
		\sin(\pi x)\cos\left(\frac{\pi y}{d}\right)\\
		&\qquad\qquad\qquad\qquad -
		\frac{2\pi^2c}{d^2}
		\sin\left(2\pi x\right)
		\cos(\pi x)\sin\left(\frac{\pi y}{d}\right)
		\bigg]\! \dd x \dd y= -\frac{32\sqrt{2}c}{3d^2}.\\
	\end{aligned} 
	\]
	It is not hard to check that all the remaining entries of the matrix $\sfA$ are zero. From these computations we infer that the matrix $\sfA$ has the block-diagonal structure
	\[
		\sfA  = \begin{pmatrix} 0 & 0 & 0\\
			0 & \sfL_{c,d} &0\\
			0 & 0 & \sfM_{c,d}
	\end{pmatrix}
	\]
	with $\sfL_{c,d}$ and $\sfM_{c,d}$ given in \eqref{eq:LM}. The inequality~\eqref{eq:Ray-Ritz} yields that the matrix $\sfA$ is non-negative and that its kernel is of dimension one.
	Hence, we get from the above block-diagonal structure of $\sfA$ that $\lm_3(\sfA) = \lm_2(\sfL_{c,d} \oplus \sfM_{c,d})$ and the claim follows from~\eqref{eq:Ray-Ritz}.
\end{proof}

\section{Main result and its proof}\label{sec:main}

In this section we prove the main result of this note, Theorem \ref{thm:mu3} below. As a preparation we state an observation on the order of the eigenvalues of $\sfL_{c,d}$ and $\sfM_{c,d}$.

\begin{lemma}\label{lem:formEstimate}
For $c \geq 0$ and $d > 0$, the lowest eigenvalues of the matrices $\sfL_{c,d}$ and $\sfM_{c,d}$ satisfy
\begin{align*}
 \lambda_1 (\sfL_{c,d}) \leq \lambda_1 (\sfM_{c,d})
\end{align*}
and, in particular, 
\begin{align*}
 \lm_2(\sfL_{c,d}\oplus \sfM_{c,d}) = \min\{
		\lm_2(\sfL_{c,d}), \lm_1(\sfM_{c,d})\}.
\end{align*}
\end{lemma}

\begin{proof}
Using the trial vector $(1,0)^\top$ we get that
	\[
		\lm_1(\sfL_{c,d}) \le \pi^2\left(1+\frac{c^2}{d^2}\right).
	\]
	The quadratic form of the matrix $\sfM_{c,d}$ evaluated on the vector $\xi = (\xi_1,\xi_2)^\top\in\dC^2$ can be estimated from below using $2 \Real (\xi_1 \overline \xi_2) \leq \frac{1}{t} |\xi_1|^2 + t |\xi_2|^2$ for $t > 0$, applied to $t = \frac{32 \sqrt{2} c}{3 \pi^2}$, giving
	\[
	\begin{aligned}
		(\sfM_{c,d}\xi,\xi)_{\dC^2} &=  \left[\pi^2\left(1+\frac{c^2}{d^2}\right) + \frac{\pi^2}{d^2}\right]|\xi_1|^2 + 4\pi^2\left(1+\frac{c^2}{d^2}\right)|\xi_2|^2  - \frac{64\sqrt{2} c}{3 d^2}\Real \left(\xi_1\ov{\xi_2}\right)\\
		&\ge \pi^2\left(1+\frac{c^2}{d^2}\right) |\xi_1|^2 + \left[ 4\pi^2\left( 1+\frac{c^2}{d^2}\right) -\frac{2048}{9\pi^2} \frac{c^2}{d^2}\right]|\xi_2|^2\\
		&\ge \pi^2\left(1+\frac{c^2}{d^2}\right) \big(|\xi_1|^2+|\xi_2|^2\big).
	\end{aligned}	
	\]
	This bound combined with the min-max principle yields that
	$\lm_1(\sfL_{c,d})\le \lm_1(\sfM_{c,d})$.
\end{proof}

Note that the eigenvalues of $\sfL_{c,d}$ and $\sfM_{c,d}$ can be computed explicitly, which we will use. We have
\begin{align}\label{eq:lambdaL}
 \lambda_2 (\sfL_{c,d}) = \frac{\pi^2 (c^2 + d^2 + 1) + \sqrt{\pi^4 (c^2 + d^2 - 1)^2 + 256 c^2}}{2 d^2}
\end{align}
and
\begin{align}\label{eq:lambdaM}
 \lambda_1 (\sfM_{c,d}) = \frac{3 \pi^2 (5 c^2 + 5 d^2 + 1) - \sqrt{9 \pi^4 (3 c^2 + 3 d^2 - 1)^2 + 8192 c^2}}{6 d^2}.
\end{align}

Next, let us recall a bound obtained by Kr\"oger, formulated here in the special case of dimension 2 and the third Neumann Laplacian eigenvalue.

\begin{theorem}{{\rm(\cite[Theorem 1]{K99})}}
	\label{thm:Kroger}
	Let $\Omg\subset\dR^2$ be a bounded, convex domain with diameter $d_\Omg$. Then
	\[
		\mu_3(\Omg) \le \left(\frac{2j_{0,1} + \pi}{d_\Omg}\right)^2
	\]
	holds, where $j_{0,1}\approx 2.4048$ is the first zero of the Bessel function $J_0$.
\end{theorem}

The perimeter and the diameter of the parallelogram $\Omg_{c,d}$ can be computed by the formulae
\[
	|\p\Omg_{c,d}| = 2\left(1+\sqrt{c^2+d^2}\right),\qquad d_{\Omg_{c,d}} = \sqrt{(1+c)^2+d^2}.
\]
Thus, the bound in Theorem~\ref{thm:Kroger} in the special case of parallelograms reads as
\begin{equation*}
% \label{eq:Kroger}
	\mu_3(\Omg_{c,d}) \le \frac{(2j_{0,1}+\pi)^2}{(1+c)^2+d^2}.
\end{equation*}

Now the bound~\eqref{eq:mu3} for parallelograms follows from the following theorem; cf.\ Proposition \ref{prop:LM}.

\begin{theorem}[Main result]\label{thm:mu3}
	For any $c \geq 0$ and $d > 0$ with $c^2 + d^2 \leq 1$ we have
	\[
		\min\left\{\lm_2(\sfL_{c,d}\oplus \sfM_{c,d}),\frac{(2j_{0,1}+\pi)^2}{(1+c)^2+d^2}\right\}\le\frac{9\pi^2}{(1+\sqrt{c^2+d^2})^2},
	\]
	and, in particular, 
	\[ 
	|\p\Omg_{c,d}|^2\mu_3(\Omg_{c,d}) \le 36\pi^2.
	\]
	Both inequalities are strict if $(c,d)\ne (0,\frac12)$.
\end{theorem}

\begin{proof}
Throughout this proof we will mostly work with polar coordinates
\begin{align*}
 c = r \cos \theta, \quad d = r \sin \theta, \qquad 0 < r \leq 1, \quad 0 < \theta \leq \frac{\pi}{2}.
\end{align*}
In these coordinates, our aim is to prove that
\begin{align}\label{eq:greatminimum}
 \min\left\{\lm_2(\sfL_{r \cos \theta, r \sin \theta}\oplus \sfM_{r \cos \theta, r \sin \theta}),\frac{(2j_{0,1}+\pi)^2}{1 + 2 r \cos \theta + r^2}\right\}<\frac{9\pi^2}{(1+r)^2}
\end{align}
holds for all $r \in (0, 1], \theta \in (0, \frac{\pi}{2}]$ such that $(r,\tt)\ne (\frac12,\frac{\pi}{2})$.

{\it Step 1: Small angles.} For
\begin{align*}
 c_\star := \frac{2(2j_{0,1}+\pi)^2}{9\pi^2}-1 \approx 0.4235
\end{align*}
consider the region
\begin{align*}
 R_1 := \left\{ (r, \theta) \in (0,1] \times (0, \tfrac{\pi}{2}] : \cos \theta > c_\star \right\}.
\end{align*}
Since 
\begin{align*}
 \frac{(1 + r)^2}{18 \pi^2 r} (2 j_{0,1} + \pi)^2 - \frac{r + \frac{1}{r}}{2} & = \left(\frac{(2j_{0,1}+\pi)^2}{18\pi^2}-\frac12\right) \left(r+\frac{1}{r}\right) + \frac{(2j_{0,1}+\pi)^2}{9\pi^2}
\end{align*}
is strictly increasing as a function of $r \in (0, 1]$, we have
\begin{align*}
 \frac{(1 + r)^2}{18 \pi^2 r} (2 j_{0,1} + \pi)^2 - \frac{r + \frac{1}{r}}{2} \le c_\star < \cos \theta
\end{align*}
for all $(r, \theta) \in R_1$, which is equivalent to
\begin{align*}
 \frac{(2j_{0,1}+\pi)^2}{1 + 2 r \cos \theta + r^2} < \frac{9\pi^2}{(1+r)^2},
\end{align*}
and, thus, implies \eqref{eq:greatminimum} whenever $(r, \theta) \in R_1$.

{\it Step 2: Large angles and radii.}
Consider the region
	\[
	R_2 := \left\{(r,\tt)\in (0,1]\times (0,\tfrac{\pi}{2}] \colon r\in( \tfrac12,1),\cos^2 \theta <  1 - \frac{(r^2+\frac34)^2}{4 r^2}, \cos\theta \le c_\star \right\}.
	\] 
We will show
\begin{align}\label{eq:finalResultStep2}
 \lm_2(\sfL_{r\cos\tt,r\sin\tt}) = \frac{\pi ^2\left(r^2+1\right)+\sqrt{\pi ^4(r^2-1)^2+256r^2\cos^2\tt}}{2 r^2\sin^2\tt} < \frac{9 \pi^2}{(1 + r)^2}
\end{align}
for all $(r, \theta) \in R_2$, cf.\ \eqref{eq:lambdaL}. First, note that $\cos^2 \theta < 1 - \frac{(r^2+\frac34)^2}{4 r^2}$  implies $\sin^2 \tt > \frac{(r^2+\frac34)^2}{4 r^2}$. Hence,
\begin{align*}
 \lm_2(\sfL_{r\cos\tt,r\sin\tt}) < \frac{\pi^2 (r^2 + 1) + \sqrt{\pi^4 (r^2 - 1)^2 + 160 r^2 - 64 r^4 - 36}}{\frac{1}{2} (r^2 + \frac{3}{4})^2}, \quad (r, \theta) \in R_2,
\end{align*}
and, thus, in order to obtain the inequality \eqref{eq:finalResultStep2} it is sufficient to prove
	\[
		\frac{9\pi^2 \big(r^2+\frac34 \big)^2}{(1 + r)^2} - 2 \pi^2(r^2+1) > 2 \sqrt{\pi^4(r^2-1)^2+160r^2-64r^4-36}
	\]
for all $r \in (\frac12,1)$. This is done in Lemma \ref{lem:schrecklich}.

{\it Step 3: Intermediate angles and radii.} Let us now consider the region
	\[
	R_3 := \left\{(r,\tt)\in (0,1]\times(0,\tfrac{\pi}{2}]\colon \frac12< r \le r_\star, 1 - \frac{(r^2+\frac34)^2}{4 r^2} \leq \cos^2 \tt \le c_\star^2 \right\},
	\]
	where $r_\star := \sqrt{1-c_\star^2}-\sqrt{\frac14-c_\star^2} \approx 0.6401$. Our aim is to prove $\lambda_1 (\sfM_{r \cos \tt,r \sin \tt}) <\frac{9 \pi^2}{(1+r)^2}$ in this region. Let us make the substitution $t :=\cos^2\tt$ in the expression for $\lm_1(\sfM_{r\cos\tt,r\sin\tt})$, see \eqref{eq:lambdaM}, and consider the function
	\begin{align}\label{eq:g}
		g(r,t) := 3\pi^2(5r^2+1)-\sqrt{9\pi^4(3r^2-1)^2+8192r^2t} - \frac{54\pi^2r^2}{(1+r)^2}(1-t).
	\end{align}
	In the described region, the first entry of the minimum in~\eqref{eq:greatminimum} is smaller than the right hand side if the inequality
	\[
		g(r,\cos^2\tt) < 0,\qquad r\in\left(\tfrac12,r_\star\right],
	\]
	holds. By computing the second derivative with respect to $t$ we infer that for any fixed $r\in(0,1)$ the function $(0,1)\ni t\mapsto g(r,t)$ is strictly convex. Hence, for any $t_1 < t < t_2$ we have $g(r,t) \le \max\{g(r,t_1),g(r,t_2)\}$. Thus for \eqref{eq:g} it suffices to check that
	\[
		\max\left\{g(r,c_\star^2), g\left(r,1-\frac{(r^2+\frac34)^2}{4r^2}\right)\right\} < 0,\qquad r\in\left(\tfrac12,r_\star\right].
	\]
	This inequality follows from Lemmata~\ref{lem:schrecklich2} and~\ref{lem:schrecklich4} in Appendix~\ref{app}.

{\it Step 4: large angles and small radii.}
	Let us now consider the region 
	\[
	R_4 := \left\{(r,\tt)\in (0,1]\times (0,\tfrac{\pi}{2}]\colon r \le \frac12, \cos\tt \le c_\star\right\}
	\]
	Also here we want to prove $\lambda_1 (\sfM_{c,d}) < \frac{9 \pi^2}{(1+r)^2}$ provided that $(r,\tt)\ne (\frac12,\frac{\pi}{2})$ and will use the function $g$ defined in \eqref{eq:g}. By strict convexity of $g$ with respect to $t$ in order to verify that $g(r,\cos^2\tt) < 0$ for all $(r,\tt)\in R_4\setminus \{(\frac12,\frac{\pi}{2})\}$ it is sufficient to check that
	\[
	\max\{g(r,c_\star^2),g(r,0)\} \le 0,\qquad r\in \left(0,\tfrac12\right],
	\]
	the inequality being strict if $r\ne\frac12$.
	This inequality follows from Lemma~\ref{lem:schrecklich2} and Lemma~\ref{lem:schrecklich3} in Appendix~\ref{app}.
	
	{\it Step 5: final step.} It remains to notice that the union of $R_1, \dots, R_4$ covers the quarter of the unit disk $(r,\tt)\in (0,1]\times(0,\frac{\pi}{2}]$ without holes (see Figure~\ref{fig2}). For this, note that for all $r \in (r_\star, 1]$ we have $1 - \frac{(r^2 + \frac{3}{4})^2}{4 r^2} > 1 - \frac{(r_\star^2 + \frac{3}{4})^2}{4 r_\star^2} = c_\star^2$. 
	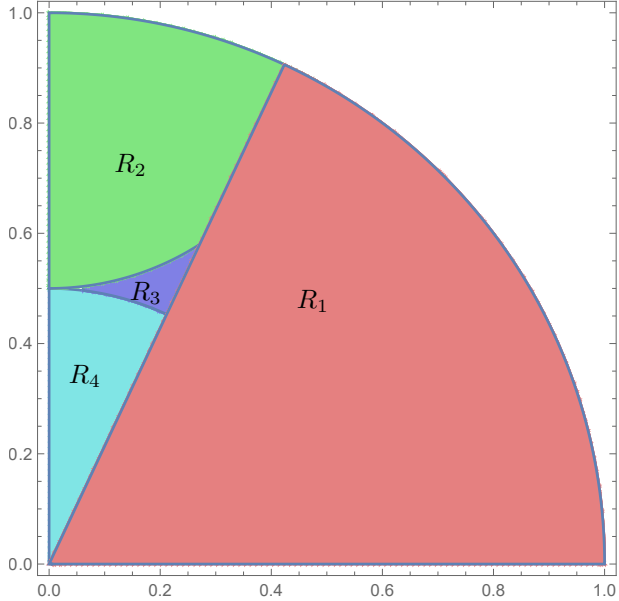
\begin{figure}[h]
		\begin{tikzpicture}
		\draw (0,0) node[inner sep=0] {\includegraphics[width=8cm]{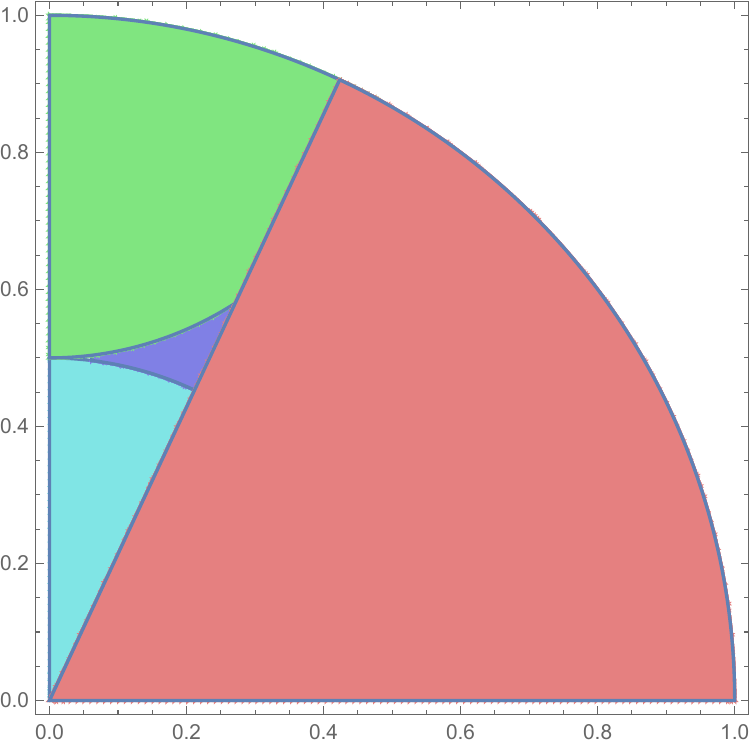}};
		\draw (0,0) node {$R_1$};
		\draw (-3,-1) node {$R_4$};
		\draw (-2.2,0.1) node {$R_3$};
		\draw (-2.4,1.8) node {$R_2$};
		\end{tikzpicture}
	\caption{The regions $R_1,R_2,R_3$, and $R_4$.}
	\label{fig2}
	\end{figure}
\end{proof}

\subsection*{Acknowledgements}
VL is grateful to Stockholm University for the possibility to have a research stay in June 2024, where this project was initiated. JR acknowledges financial support by the grant no.\ 2022-03342 of the Swedish Research Council (VR).

\appendix

\section{Auxiliary inequalities}\label{app}

In this appendix, we collect some auxiliary inequalities which are used in the proof of the main result, Theorem~\ref{thm:mu3}.

\begin{lemma}\label{lem:schrecklich}
The inequality
\begin{align}\label{eq:ungetuem}
 \frac{9\pi^2 \big(r^2+\frac34 \big)^2}{(1 + r)^2} - 2 \pi^2(r^2+1) > 2 \sqrt{\pi^4(r^2-1)^2+160r^2-64r^4-36}
\end{align}
holds for all $r > \frac12$.
\end{lemma}

\begin{proof}
Note first that the left-hand side of the inequality \eqref{eq:ungetuem} is positive for all $r > \frac{1}{2}$ since
\begin{align*}
 9 (r^2 + \tfrac34)^2 - 2 (r^2 + 1) (1 + r)^2 = \frac{27}{8} + 6 s + 14 s^2 + 10 s^3 + 7 s^4,
\end{align*}
where $s = r - \tfrac12 > 0$. Thus \eqref{eq:ungetuem} follows if we can show its squared version
\begin{align*}
 \frac{81 \pi^4 (r^2 + \tfrac34)^4}{(1+r)^4} - 36 \pi^4 \frac{(r^2 + \tfrac34)^2 (r^2 + 1)}{(1+r)^2} & + 4 \pi^4 (r^2 + 1)^2 \\
 & > 4 \big( \pi^4(r^2-1)^2+160r^2-64r^4-36 \big),
\end{align*}
which is equivalent to
\begin{align}\label{eq:nochSchlimmer}
\begin{split}
 81 \pi^4 (r^2 + \tfrac34)^4 - 36 \pi^4 & (r^2 + \tfrac34)^2 (r^2 + 1) (1+r)^2 + 4 \pi^4 (r^2 + 1)^2 (1 + r)^4 \\
 & - 4 \big( \pi^4(r^2-1)^2+160r^2-64r^4-36 \big) (1 + r)^4 > 0
\end{split}
\end{align}
for $r > \frac12$. Furthermore, by collecting the terms with respectively without the factor $\pi^4$ and factorizing one sees that the left-hand side of \eqref{eq:nochSchlimmer} equals
\begin{align*}
 \frac{2 r - 1}{256} \Big( 4096 & (1+r)^4 (2r-3) (2r+1) (2r+3) \\
 & + \pi^4 (6 r - 1) (4 r^2 - 8 r + 9) (12r^2 +8r + 17) (20 r^2 + 8 r + 9) \Big).
\end{align*}
Since $4 r^2 - 8 r + 9 > 0$ holds for all real $r$ and $\pi^4 > 96 = 3 \cdot 32$, to obtain \eqref{eq:nochSchlimmer} it suffices to show $p (r) > 0$ for all $r > \frac12$, where
\begin{align*}
 p (r) & = 128 (1+r)^4 (2r-3) (2r+1) (2r+3) \\
 & \quad + 3 (6 r - 1) (4 r^2 - 8 r + 9) (12r^2 +8r + 17) (20 r^2 + 8 r + 9) \\
 & = 32 \big( 162 + 1269 s + 2502 s^2 + 3759 s^3 + 3680 s^4 + 3136 s^5 + 1552 s^6 + 572 s^7 \big)
\end{align*}
for $s = r - \frac12$. Especially, for $r > \frac{1}{2}$ we have $s > 0$ and, hence, $p (r) > 0$. This completes the proof of the lemma.
\end{proof}

For the remainder of this appendix, recall that the constants $c_\star$ and $r_\star$ are given by 
\[
	c_\star = \frac{2(2j_{0,1}+\pi)^2}{9\pi^2}-1 \approx 0.4235,\qquad r_\star = \sqrt{1-c_\star^2} - \sqrt{\frac14-c_\star^2} \approx 0.6401.
\]

\begin{lemma}\label{lem:schrecklich2}
	The inequality
	\[
		3\pi^2(5r^2+1)-\sqrt{9\pi^4(3r^2-1)^2+8192r^2c_\star^2} - \frac{54\pi^2r^2}{(1+r)^2}(1-c_\star^2)< 0
	\]
	holds for all $r\in(0,1]$.
\end{lemma}

\begin{proof}
	It suffices to check that for all $r\in(0,1]$
	\[
	\left(3\pi^2(5r^2+1)-\frac{54\pi^2r^2}{(1+r)^2}(1-c^2_\star)\right)^2< 9\pi^4(3r^2-1)^2+8192r^2c^2_\star.
	\]
	After multiplying the inequality by $(1+r)^4$ we end up with the necessity to show that for all $r\in(0,1]$
	\begin{align*}
	 \big(9\pi^4(3r^2-1)^2 + & \, 8192r^2c^2_\star \big) (1 + r)^4 \\
	 & - \left( 3\pi^2(5r^2+1) (1 + r)^2 - 54\pi^2r^2(1-c^2_\star) \right)^2 > 0.
	\end{align*}
	Simplifying the left-hand side and dividing it by $r^2$ leads to the requirement to show that the polynomial
	\begin{align*}
	 G (r) := 324 \pi^4 (1 - c_\star^2) (5 r^2 + 1) & (1 + r)^2 - 144 \pi^4 (1 + r^2) (1 + r)^4 \\
	 & - 2916 \pi^4 r^2 (1 - c_\star^2)^2 + 8192 c_\star^2 (1 + r)^4
	\end{align*}
	is positive for $r \in (0, 1]$. In fact, by a direct computation we get that $G(0) > 0$. Suppose for the moment that $G$ is not positive on the whole interval $(0,1]$. Then there exists $r_0 \in (0,1]$ such that $G(r_0) = 0$. In this case one of the following $r_0 \in (0,\frac12]$, $r_0\in (\frac12,\frac34]$,
	or $r_0\in (\frac34,1]$ holds. By a direct, but cumbersome computation, of the integrals we verify the following inequalities\footnote{Approximate numerical values of the ratios in the left-hand sides are 7.28761, 29.7613 and 9.11554, respectively.}
	\begin{equation}\label{eq:Ginequalities}
		\frac{\displaystyle\int_0^{\frac12}(G'(r))^2\dd r}{\displaystyle\int_0^{\frac12}(G(r))^2\dd r} < \pi^2,\quad
		\frac{\displaystyle\int_{\frac12}^{\frac34}(G'(r))^2\dd r}{\displaystyle\int_{\frac12}^{\frac34}(G(r))^2\dd r} < 4\pi^2,
		\quad\text{and}\quad
		\frac{\displaystyle\int_{\frac34}^1(G'(r))^2\dd r}{\displaystyle\int_{\frac34}^1(G(r))^2\dd r} < 4\pi^2.
	\end{equation} 
	If $r_0\in(0,\frac12]$, then we get by the min-max principle that for any non-trivial $f\in H^1(0,\frac12)$
	with  $f(r_0) = 0$ there holds
	\[
		\frac{\displaystyle\int_0^{\frac12}|f'(r)|^2\dd r}{\displaystyle\int_0^{\frac12}|f(r)|^2\dd r} \ge \pi^2
	\]	
	leading to a contradiction with the first inequality in~\eqref{eq:Ginequalities}.
	If $r_0\in(\frac12,\frac34]$, then we get by the min-max principle that for any non-trivial $f\in H^1(\frac12,\frac34)$
	with  $f(r_0) = 0$ there holds
	\[
	\frac{\displaystyle\int_{\frac12}^{\frac34}|f'(r)|^2\dd r}{\displaystyle\int_{\frac12}^{\frac34}|f(r)|^2\dd r} \ge 4\pi^2
	\]	
	leading to a contradiction with the second inequality in~\eqref{eq:Ginequalities}. Analogously, we exclude the case when $r_0\in(\frac34,1]$ by a contradiction to the third inequality in~\eqref{eq:Ginequalities}. Thus, $G$ does not vanish on $(0,1]$, by which the argument is complete.
\end{proof}

\begin{lemma}\label{lem:schrecklich3}
	The inequality
	\[
	3\pi^2(5r^2+1)-\sqrt{9\pi^4(3r^2-1)^2} - \frac{54\pi^2r^2}{(1+r)^2}< 0
	\]
	holds for all $r\in\left(0,\frac12\right)$.
\end{lemma}

\begin{proof}
	It suffices to check that for all $r\in(0,\frac12)$
	\[
	\left(3\pi^2(5r^2+1)-\frac{54\pi^2r^2}{(1+r)^2}\right)^2< 9\pi^4(3r^2-1)^2.
	\]
	Expanding the left-hand side, combining like terms and multiplying the inequality by $\frac{(1+r)^4}{\pi^4r^2}$ we end up with the necessity to show that for all $r\in(0,\frac12)$
	\[
9(5r^2+1)(1+r)^2-4(1+r)^4(1+r^2)-81 r^2> 0.
	\]
	Let us define the polynomial
	\[
	\begin{aligned}
	H(r)&:=9(5r^2+1)(1+r)^2-4(1+r)^4(1+r^2)-81 r^2 \\
	&=(1-2 r) (2 r+5) \left(r^4+2 r^3-7 r^2+2 r+1\right).
	\end{aligned}
	\]
	We need to show that $H(r) > 0$ for all $r\in(0,\frac12)$. The factors $1-2r$ and $2r+5$ are clearly positive for all $r\in(0,\frac12)$.
	For the third factor we get that for all $r\in(0,\frac12)$
	\[
	r^4+2 r^3-7 r^2+2 r+1 \ge -7 r^2+2 r+1 = -7\left(
	r+\frac{2\sqrt{2}-1}{7}\right)\left(
	r-\frac{1+2\sqrt{2}}{7}\right) > 0,
	\]
	where we used that $\frac{1+2\sqrt{2}}{7} > \frac12$.
	Thus, we conclude that $H(r) > 0$ for all $r\in(0,\frac12)$. 
\end{proof}

\begin{lemma}\label{lem:schrecklich4}
	The inequality
	\[
	3 \pi ^2 \left(5 r^2+1\right)-\sqrt{9 \pi ^4 \left(3r^2-1\right)^2-128 \left(16 r^4-40 r^2+9\right)}-\frac{27 \pi ^2 \left(r^2+\frac{3}{4}\right)^2}{2 (r+1)^2} < 0
	\]
	holds for all $r\in\left(\frac12,r_
	\star\right]$.
\end{lemma}

\begin{proof}
Consider the function
\[
\begin{aligned}
	S(r) := 3 \pi ^2 \left(5 r^2+1\right)- & \, \sqrt{9 \pi ^4 \left(3r^2-1\right)^2-128 \left(16 r^4- 40 r^2+9\right)} \\
	& \qquad \qquad \qquad \qquad \qquad \qquad -\frac{27 \pi ^2 \left(r^2+\frac{3}{4}\right)^2}{2 (r+1)^2}.
\end{aligned}
\]
The graph of $S$ is plotted in Figure~\ref{graph}.
\begin{figure}[h]
	\includegraphics[width=8cm]{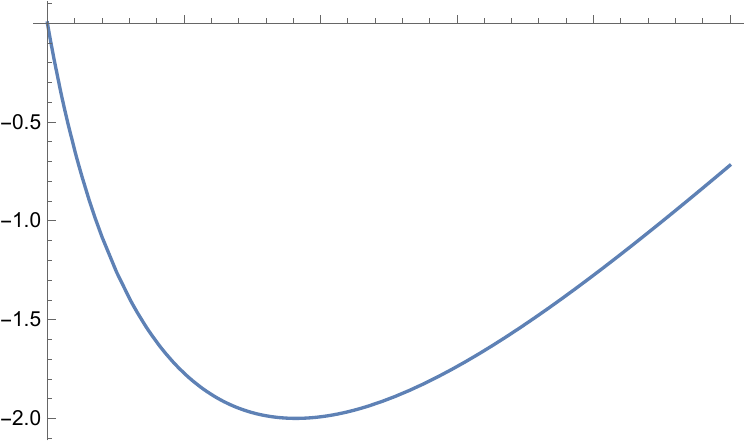}
	\caption{Graph of the function $S$ on the interval $[\frac12,\frac34]$}
	\label{graph}
\end{figure}
Notice that 
\[
	S\left(\frac12\right) = 0,\qquad S\left(r_\star\right)\approx -1.81 < 0.
\]
It suffices to show that $S''(r) > 0$ for all $r\in(\frac12,r_\star)$. Let us introduce the auxiliary functions
\[
\begin{aligned}
	&S_1(r) := -\frac{27 \pi ^2 \left(r^2+\frac{3}{4}\right)^2}{2 (r+1)^2},
	\quad S_2(r) := 3 \pi ^2 \left(5 r^2+1\right),\\
	&S_3(r) := -\sqrt{9 \pi ^4 \left(3 r^2-1\right)^2-128 \left(16 r^4-40 r^2+9\right)},
\end{aligned}	
\]
so that the representation $S = S_1 + S_2 + S_3$ holds.
The second and the third derivatives of $S_1$ are given by
\[
	S''_1(r) = -\frac{27 \pi ^2 \left(16 r^4+64 r^3+96 r^2-48 r+51\right)}{16 (r+1)^4},\qquad S_1'''(r) = -\frac{567 \pi ^2 (4 r-3)}{4 (r+1)^5}.
\]
Hence, for all $r\in(\frac12,r_\star)$ we have $S'''_1(r) > 0$ and thus $S''_1(r) > S''_1(\frac12) = -20\pi^2$.
The second and the third derivatives of $S_3$ are as follows 
\[
\begin{aligned}
	S_3''(r) &= \frac{-486 \pi ^8 \left(3 r^2-1\right)^3-131072 \left(64 r^6-240 r^4+108 r^2-45\right)}{\left(9 \pi ^4 \left(1-3 r^2\right)^2-128 \left(16 r^4-40 r^2+9\right)\right)^{3/2}} \\
	& \qquad \qquad + \frac{2304 \pi ^4 \left(288 r^6-684 r^4+291 r^2-47\right)}{\left(9 \pi ^4 \left(1-3 r^2\right)^2-128 \left(16 r^4-40 r^2+9\right)\right)^{3/2}},\\
	S_3'''(r) &= -\frac{1536 \left(32768-207 \pi ^4\right) r \left((81 \pi^4 - 2048) r^4 + 1152 - 9 \pi^4 \right)}{\left(9 \pi ^4 \left(1-3 r^2\right)^2-128 \left(16 r^4-40 r^2+9\right)\right)^{5/2}}.
\end{aligned}	
\]
Using that $32768-207\pi^4 > 0$, $81\pi^4-2048 > 0$, and $1152 - 9 \pi^4 > 0$, we conclude that $S_3'''(r) < 0$ for all $r\in(\frac12,r_\star)$. Thus, we get that $S_3''(r) > S_3''(r_\star) \approx -22.28$ for all $r\in[\tfrac12,r_\star)$. It remains to notice that for all $r\in[\frac12,r_\star)$
\[
	S''(r) = S_1''(r) + S_2''(r) +S_3''(r) > -20\pi^2+30\pi^2+S_3''(r_\star) > 0.
\]
This completes the proof of the lemma.
\end{proof}


\begin{thebibliography}{99}


%
\bibitem{ABLW09}
M.~Ashbaugh (ed.), R.~Benguria (ed.), R.~Laugesen (ed.), and T.~Weidl (ed.),
Low Eigenvalues of Laplace and Schr\"{o}dinger Operators,
Oberwolfach Rep. {6}  (2009), 355--428.
%
\bibitem{BHM24}
B.~Bogosel, A.~Henrot, and M.~Michetti, {\it Optimization of Neumann Eigenvalues under convexity and geometric constraints}, SIAM J. Math. Anal. 56\ (2024),  no. 6, 7327--7349. 

\bibitem{BH19} D.\ Bucur and A.\ Henrot, {\it Maximization of the second non-trivial Neumann eigenvalue}, Acta Math.\ 222 (2019), no.\ 2, 337--361.

\bibitem{CL24}
J.~Clutterbuck and J.~Larsen-Scott, 
{\it Local spectral optimisation for Robin problems with negative boundary parameter on quadrilaterals},
{J.\ Math.\ Phys.} {65} (2024), 031505.

\bibitem{FLW}
R.\,L.~Frank, A.~Laptev, and T.~Weidl, 
{Schr\"{o}dinger Operators: Eigenvalues and Lieb-Thirring Inequalities}, Cambridge University Press, Cambridge, 2023.

\bibitem{GNP09} A.\ Girouard, N.\ Nadirashvili, and I.\ Polterovich, {\it Maximization of the second positive Neumann eigenvalue for planar domains}, J.\ Differential Geom.\ 83 (2009), no.\ 3, 637--661.

\bibitem{H06} A.\ Henrot, Extremum Problems for Eigenvalues of Elliptic Operators, Front.\ Math., Birkh\"auser Verlag, Basel, 2006.

\bibitem{HLL24} A.\ Henrot, A.\ Lemenant, and I.\ Lucardesi, {\it An isoperimetric problem with two distinct solutions}, Trans.\ Amer.\ Math.\ Soc.\   377 (2024), no. 4, 2367--2411.
% %

\bibitem{KLV23}
D.~Krej\v{c}i\v{r}\'{i}k, V.~Lotoreichik, and T.~Vu, {\it Reverse isoperimetric inequality for the lowest Robin eigenvalue of a triangle}, {Appl.\ Math.\ Optim.} {88} (2023), Paper No.\ 63.

\bibitem{K99} P.~Kr\"oger, {\it On upper bounds for high order Neumann eigenvalues of convex domains in Euclidean space}, Proc.\ Amer.\ Math.\ Soc.\ 127 (1999), no. 6, 1665--1669. 

\bibitem{LS09} R.\,S.\ Laugesen and B.\,A.\ Siudeja, {\it Maximizing Neumann fundamental tones of triangles}, J.\ Math.\ Phys.\ 50 (2009), no. 11, 112903, 18 pp.

\bibitem{LS17} R.\,S.\ Laugesen and B.\,A.\ Siudeja, {\it Triangles and other special domains}, Shape Optimization and Spectral Theory, 149--200, De Gruyter Open, Warsaw, 2017.

\bibitem{LR24} C.~L\'{e}na and J.~Rohleder, {\it Estimates for the lowest Neumann eigenvalues of parallelograms and domains of constant width}, {Anal.\ Math.\ Phys.} {14}  (2024), Paper No.\ 42.



\end{thebibliography}
\end{document}